\documentclass[11pt]{amsart}
\usepackage[margin=1.10in]{geometry}
\usepackage{amscd,amsmath,amsxtra,amsthm,amssymb,stmaryrd,xr,mathrsfs,mathtools,enumerate,commath, comment, enumitem}
\usepackage{stmaryrd}
\usepackage{xcolor}
\usepackage{commath}
\usepackage{comment}
\usepackage{tikz-cd}
\usepackage{longtable} 
\usepackage{pdflscape} 
\usepackage{booktabs}
\usepackage{hyperref}
\definecolor{vegasgold}{rgb}{0.77, 0.7, 0.35}
\definecolor{darkgoldenrod}{rgb}{0.72, 0.53, 0.04}
\definecolor{gold(metallic)}{rgb}{0.83, 0.69, 0.22}
\hypersetup{
 colorlinks=true,
 linkcolor=darkgoldenrod,
 filecolor=brown,      
 urlcolor=gold(metallic),
 citecolor=darkgoldenrod,
 }

\newtheorem{lthm}{Theorem}

\usepackage[all,cmtip]{xy}

\DeclareFontFamily{U}{wncy}{}
\DeclareFontShape{U}{wncy}{m}{n}{<->wncyr10}{}
\DeclareSymbolFont{mcy}{U}{wncy}{m}{n}
\DeclareMathSymbol{\Sh}{\mathord}{mcy}{"58}
\usepackage[T2A,T1]{fontenc}
\usepackage[OT2,T1]{fontenc}

\newtheorem{theorem}{Theorem}[section]
\newtheorem{lemma}[theorem]{Lemma}

\newtheorem{definition}[theorem]{Definition}

\newtheorem{example}[theorem]{Example}
\newtheorem{conjecture}[theorem]{Conjecture}
\newtheorem{proposition}[theorem]{Proposition}

\newcommand{\ac}{\operatorname{ac}}

\newcommand{\cK}{\mathcal{K}}

\newcommand{\bT}{\mathbf{T}}

\newcommand{\cH}{\mathcal{H}}

\newcommand{\Z}{\mathbb{Z}}
\newcommand{\p}{\mathfrak{p}}
\newcommand{\Q}{\mathbb{Q}}
\newcommand{\F}{\mathbb{F}}

\newcommand{\cL}{\mathcal{L}}

\newcommand{\cO}{\mathcal{O}}
\newcommand{\cS}{\mathcal{S}}

\newcommand{\Sel}{\op{Sel}}

\newcommand{\cyc}{\op{cyc}}

\makeatletter
\newcommand{\mylabel}[2]{#2\def\@currentlabel{#2}\label{#1}}
\makeatother

\newcommand{\op}[1]{\operatorname{#1}}

\begin{document}

\title[Mazur's growth number conjecture]{Mazur's growth number conjecture and congruences}

\author[A.~Ray]{Anwesh Ray}
\address[Ray]{Chennai Mathematical Institute, H1, SIPCOT IT Park, Kelambakkam, Siruseri, Tamil Nadu 603103, India}
\email{anwesh@cmi.ac.in}


\maketitle

\begin{abstract}
 Motivated by the work of Greenberg–Vatsal \cite{greenbergvatsal00} and Emerton–Pollack–Weston \cite{emertonpollackweston}, I investigate the extent to which Mazur’s conjecture on the growth of Selmer ranks in $\mathbb{Z}_p$-extensions of an imaginary quadratic field persists under $p$-congruences between Galois representations. As a first step, I establish Mazur’s conjecture for certain triples $(E, K, p)$ under explicit hypotheses. Building on this, I prove analogous results for Greenberg Selmer groups attached to modular forms that are congruent mod $p$ to $E$, including all specializations arising from Hida families of fixed tame level. In particular, I show that the Mordell–Weil ranks in non-anticyclotomic $\mathbb{Z}_p$-extensions of $K$ remain bounded for elliptic curves $E'$ such that $E[p]$ and $E'[p]$ are isomorphic as Galois modules.
\end{abstract}

\section{Introduction}

\subsection{Background and motivation}
Let $L$ be a number field and $E_{/L}$ be an elliptic curve. The group of $L$-rational points on $E$, denoted $E(L)$, forms an abelian group under a natural geometric addition law arising from the chord-tangent construction on the curve. This group is referred to as the Mordell--Weil group of $E$ over $L$. A foundational result in the arithmetic of elliptic curves, known as the Mordell–Weil theorem, asserts that $E(L)$ is a finitely generated abelian group. Consequently, one may write $E(L) \simeq \mathbb{Z}^r \oplus T$, where $T \subseteq E(L)$ is a finite abelian subgroup consisting of torsion points, and $r:=\operatorname{rank} E(L)$. This rank is a central arithmetic invariant of the curve and has connections to the behavior of $L$-functions. Understanding how this rank varies in families of number fields is a major theme in modern number theory.

\par In a seminal development, Mazur \cite{MazurmainIwasawa} initiated the study of the Iwasawa theory of elliptic curves, which seeks to understand the growth properties of $p$-primary Selmer groups of abelian varieties as one ascends pro-$p$ towers of number fields. These Selmer groups interpolate information about the Mordell–Weil group and $p$-primary Tate–Shafarevich group.

\par Let $\Q_{\mathrm{cyc}}$ denote the unique $\Z_p$-extension of $\Q$ and $E_{/\Q}$ be an elliptic curve. Kato \cite{Kato2004} and Rohrlich \cite{Rohrlich} show that the group $E(\Q_{\mathrm{cyc}})$ has finite rank. Let $K$ now be an imaginary quadratic field, and let $\cK$ be a $\Z_p$-extension of $K$. For each integer $n \geq 0$, denote by $\cK_n \subset \cK$ the unique subextension with $[\cK_n : K] = p^n$, so that one obtains a tower of number fields:
$$
K = \cK_0 \subset \cK_1 \subset \cdots \subset \cK_n \subset \cK_{n+1} \subset \cdots \subset \cK = \bigcup_{n \geq 0} \cK_n.
$$
\noindent Among such extensions, a particularly important one is the anticyclotomic $\Z_p$-extension of $K$, denoted $\cK_{\mathrm{ac}}$, characterized by the action of complex conjugation on its Galois group being nontrivial. A distinction is made between two types of pairs $(E,K)$. The pair is said to be exceptional if the elliptic curve $E$ has complex multiplication (CM) by an order in the imaginary quadratic field $K$; otherwise, it is termed generic. The sign of the pair $(E,K)$ plays a crucial role in formulating growth predictions. In the generic case, this sign is defined as the root number (i.e., the sign of the functional equation) of the Hasse–Weil $L$-function $L(s, E_{/K})$. In the exceptional case, when $E$ has CM by $K$, there exists a Hecke character $\varphi$ such that $L(s, E/\Q) = L(s, \varphi)$, and the sign of $(E,K)$ is then defined to be the sign of the functional equation of $L(s, \varphi)$.

\par In this context, Mazur \cite[p.~201]{Mazurmodcurvesandarith} formulated a precise conjecture describing the asymptotic growth of the $p^\infty$-Selmer groups over the finite layers $\cK_n$ of a $\Z_p$-extension $\cK$ of $K$.

\begin{conjecture}[Mazur's Growth Number Conjecture]\label{Mazur-GNC}
Let $E/\Q$ be an elliptic curve, $K$ an imaginary quadratic field, and $p$ a prime at which $E$ has good ordinary reduction. Suppose $\cK/K$ is a $\Z_p$-extension in which all primes of bad reduction for $E$ remain finitely decomposed. Then for all sufficiently large $n$, the $\Z_p$-corank of the $p^\infty$-Selmer group satisfies

$$
\operatorname{corank}_{\Z_p} \Sel_{p^\infty}(E/\cK_n) = c p^n + O(1),
$$

where the constant $c$ depends on the nature of the extension and the sign of $(E,K)$, and is given explicitly by:

\[c =
\begin{cases}
    0 & \text{if } \cK \ne \cK_{\mathrm{ac}} \text{ or } (E,K) \text{ has sign } +1, \\
    1 & \text{if } \cK = \cK_{\mathrm{ac}},\ (E,K) \text{ is generic, and has sign } -1, \\
    2 & \text{if } \cK = \cK_{\mathrm{ac}},\ (E,K) \text{ is exceptional, and has sign } -1.
\end{cases}\]

\end{conjecture}

\subsection{Prior results}
\par I present below a chronological and non-exhaustive list of results related to Conjecture \ref{Mazur-GNC}.

\begin{enumerate}
\item[(i)] The exceptional case of Conjecture \eqref{Mazur-GNC} has been extensively studied and is now well understood. A summary of the main results can be found in \cite[Section 1, Theorems 1.7 and 1.8]{GreenbergITEC}.
\item[(ii)] Recently, there has been progress on Conjecture \ref{Mazur-GNC} for elliptic curves $E$ of rank at most one. Gajek--Leonard, Hatley, Kundu and Lei \cite{GHKL} show that, under a certain semisimplicity hypothesis \eqref{hyp:SC} on the structure of the Selmer group over the \( \Z_p^2 \)-extension of \( K \), the conjecture holds.

\item[(iii)] The connection between Mazur’s growth number conjecture and Hilbert’s tenth problem was investigated by Katharina Müller and the author in \cite{muller2024hilbert}. There, we established the constancy of Mordell–Weil ranks for elliptic curves in many $\Z_p$-extensions of imaginary quadratic fields. These results, obtained through techniques from arithmetic statistics, led to the construction of infinite families of number fields over which Hilbert’s tenth problem has a negative solution. Precise statements can be found in Theorems A, B, 4.13, and 4.14 of \emph{loc. cit.} Subsequently, a solution to Hilbert’s tenth problem for number rings was announced by Koymans and Pagano \cite{koymans2024hilbert}. An independent proof of the same result was later made public by Alpöge, Bhargava, Ho, and Schnidman in \cite{alpoge2025rank}.
\item[(iv)] Using a formula for due to Disegni \cite{disegni1, Disegni2}, Kundu and Lei \cite{kundulei2025} showed that for elliptic curves $E$ of rank $1$ satisfying the generalized Heegner hypothesis \eqref{GHH}, the Conjecture \ref{Mazur-GNC} does follow from one divisibility of the Iwasawa main conjecture of $E$ over the $\Z_p^2$-extension of $K$ and a hypothesis $\langle z_{\op{Heeg}}, z_{\op{Heeg}}\rangle\neq 0$ arising from Perrin--Riou's anticyclotomic $\Lambda$-adic pairing. I refer to Theorem \ref{kundu lei theorem} for the precise statement of their result. Examples verifying $\langle z_{\op{Heeg}}, z_{\op{Heeg}}\rangle\neq 0$ have not been computed in \emph{loc. cit.}, however in principle, can be computed via the approach outlined in \cite{heightsofheegnerpoints}. 

\end{enumerate}

\subsection{Main results}
\par Let $E$ and $E'$ be elliptic curves over $\Q$ with good ordinary reduction at a prime $p$. One says that $E$ and $E'$ are \emph{$p$-congruent} if their mod $p$ Galois representations $E[p]$ and $E'[p]$ are isomorphic as $\operatorname{Gal}(\overline{\Q}/\Q)$-modules. In this case, Greenberg and Vatsal \cite{greenbergvatsal00} demonstrated that the Iwasawa invariants associated to $E$ are closely related to those of $E'$. Analogous phenomena have been investigated by Emerton, Pollack, and Weston \cite{emertonpollackweston} in the context of $p$-ordinary Hida families of modular forms. The aim of this paper is to investigate the propagation of Mazur’s growth number conjecture under congruences between Galois representations. Let $E/\Q$ be an elliptic curve and let $f$ be a Hecke eigenform. Suppose $p$ is a prime such that the residual Galois representation associated to $f$ is isomorphic to the Galois representation on $E[p]$. Assuming that Conjecture \ref{Mazur-GNC} holds for the triple $(E, K, p)$, I ask: what conclusions, if any, can be drawn about the corresponding triple $(f, K, p)$, particularly concerning the growth of $\Z_p$-coranks of Selmer groups in various $\Z_p$-extensions of $K$?

\par In the special case where $f$ has weight $2$, trivial nebentypus, and rational Fourier coefficients, it corresponds to the isogeny class of an elliptic curve $E'/\Q$. I consider the situation in which the residual representation is absolutely irreducible and the associated elliptic curve $E'$ is fixed. In this setting, an isomorphism $E[p] \simeq E'[p]$ of $\operatorname{Gal}(\overline{\Q}/\Q)$-modules has significant consequences for Conjecture \ref{Mazur-GNC} as it pertains to the triple $(E', K, p)$. In order for my congruence results to apply, it becomes necessary that the $\mu$-invariant of the Selmer groups vanish in $\Z_p$-extensions of number fields. The $p$-adic regulator of $E_{/K}$ is denoted by $\operatorname{Reg}_p(E/K)$. It is conjectured that the $p$-adic regulator is non-zero (see \cite{Schneider85}). Let $\mathcal{R}_p(E/K)$ represent the normalized regulator, defined as $\mathcal{R}_p(E/K) := \operatorname{Reg}_p(E/K)/p$. Given a prime $v$ of $K$, denote by $c_v(E/K)$ the Tamagawa number of $E$ over $K_v$. Set $\kappa_v$ to denote the residue field at $v$. Suppose $E$ has good reduction at $v$, set $\widetilde{E}(\kappa_v)$ to be the group of $\kappa_v$-rational points of the reduction of a regular model of $E$ at the prime $v$. \begin{lthm}[Proposition \ref{main long propn}]\label{thmA}
    Let $E$ be an elliptic curve over $\Q$ with good ordinary reduction at a prime $p\geq 5$ and $K$ be an imaginary quadratic number field. Let $N$ be the conductor of $E$. Assume that the following conditions are satisfied for $(E,K, p)$:
    \begin{enumerate}
        \item $E(K)[p]=0$,
        \item all rational primes dividing $Np$ are split in $K$, 
        \item $\op{corank}_{\Z_p} \op{Sel}_{p^\infty} (E/K)=1$,
        \item $p\nmid \mathcal{R}_p(E/K)$,
        \item $\Sh(E/K)[p^\infty]=0$,
        \item $p\nmid c_v(E/K)$ for all primes $v\nmid p$ of $K$,
        \item and $\widetilde{E}(\kappa_v)[p^\infty]=0$ for all primes $v|p$ of $K$.
        \item the $2$-variable Iwasawa main conjecture holds (see \eqref{IMC}),
    \end{enumerate}
 Let $\cK$ be a $\Z_p$-extension of $K$ in which the primes of $K$ that lie above $p$ are ramified. Then, the following assertions hold.
    \begin{enumerate}
        \item[(a)] Assume that $\cK$ is not the anticyclotomic extension $K_{\rm{ac}}$. The Selmer group $\op{Sel}_{p^\infty}(E/\cK)$ is cotorsion over the Iwasawa algebra with $\mu$-invariant equal to $0$ and $\lambda$-invariant equal to $1$. 
        \item[(b)] For $\cK\neq K_{\rm{ac}}$, the corank of $\op{Sel}_{p^\infty}(E/\cK_n)$ is $1$ for all $n\geq 1$.
        \item[(c)] Conjecture \ref{Mazur-GNC} holds for $(E,K,p)$.
    \end{enumerate}
\end{lthm}
\noindent I note that (b) follows easily from (a), due to inequalities 
\[1=\op{corank}_{\Z_p}\op{Sel}_{p^\infty}(E/K)\leq \op{corank}_{\Z_p}\op{Sel}_{p^\infty}(E/\cK_n)\leq \op{corank}_{\Z_p}\op{Sel}_{p^\infty}(E/\cK)=1.\] 
Let $v$ be a prime of $K$ which lies above $p$. By class field theory, there is a unique $\Z_p$-extension of $K$ in which $v$ is unramified. Hence, there are precisely three $\Z_p$-extensions of $K$ for which the conclusions (a) and (b) of the above Theorem does not apply.
\begin{lthm}[Theorem \ref{main theorem of paper}]\label{thm B}
Let $(E, K, p)$ be as in Theorem \ref{thmA}. Let $f$ be a $p$-ordinary Hecke eigencuspform with conductor $N_f$ which is coprime to $p$. Assume that:
\begin{enumerate}
    \item all primes $\ell|N_f$ are split in $K$, 
    \item the residual representation of $f$ is absolutely irreducible and is isomorphic as a $\op{G}_{\Q}$-module to (a base change of) $E[p]$.
\end{enumerate}
 Let $\cK$ be a $\Z_p$-extension of $K$ in which the primes of $K$ that lie above $p$ are ramified. Assume that $\cK$ is not the anticyclotomic extension $K_{\rm{ac}}$. Then the $\Z_p$-corank of the Greenberg Selmer group of $f$ over $\cK$ is finite. In particular, if $f$ coincides with an elliptic curve $E'$, then the rank of $E'(\cK_n)$ is bounded as $n\rightarrow \infty$.
\end{lthm}
\noindent Let $E_{/\Q}$ be an elliptic curve with good ordinary reduction at $p$ and level $N$, and let $k$ be any positive integer with $k\equiv 2\pmod{(p-1)}$. Assume that the residual representation on $E[p]$ is absolutely irreducible. Then there is a $p$-ordinary modular form of weight $k$, level $N$ such that condition (2) of Theorem \ref{thm B} is satisfied. This modular form arises as a weight $k$ specialization of a Hida family of tame level $N$ associated to the residual representation of $E$. Note that if $E$ satisfies the Heegner hypothesis, then so does $f$. Therefore, if $E$ is an elliptic curve satisfying the conditions of Theorem \ref{thmA}, then for any $k\geq 2\pmod{p-1}$, there is a modular form $f$ of weight $k$ for which the conclusion of Theorem \ref{thm B} applies. This application is discussed further in Section \ref{hida subsection}. Congruences between modular forms of fixed weight and varying level are ubiquitous, as established by the standard level-raising results of Diamond and Taylor \cite{diamond1994non}. \par An explicit example of two $5$-congruent elliptic curves $E$ and $E'$ satisfying the conditions of Theorem \ref{thm B} is given in Example \ref{example}. The relevant conditions here are fully explicit and can be easily verified using standard computational packages such as \emph{SAGE} or \emph{MAGMA}.
\subsection*{Acknowledgement} The author completed the first draft of this article while participating in the program on \hyperref[https://icts.res.in/program/afbk]{\emph{Automorphic Forms and the Bloch–Kato Conjecture}} at the International Centre for Theoretical Sciences (ICTS), Bangalore. He is grateful to ICTS for its hospitality and for providing a stimulating and supportive research environment during his visit. The author thanks Antonio Lei and Katharina M\"uller for helpful comments.

\subsection*{Conflict of interest} The author reports that there are no conflicts of interest to declare.

\subsection*{Data Availability} There is no data associated to the results of this manuscript.

\section{Iwasawa theory of Selmer groups and congruences}

\subsection{Selmer groups over $\Z_p$-extensions}
\par This section introduces preliminary notions and fixes standard notation, chosen to be consistent with that of \cite{GHKL, muller2024hilbert, kundulei2025} whenever possible. I fix an algebraic closure $\bar{\Q}$ of $\Q$ (resp. $\bar{\Q}_\ell$ of $\Q_\ell$) and for each prime $\ell$, choose an embedding $\iota_\ell: \bar{\Q}\hookrightarrow \bar{\Q}_\ell$. Note that $\iota_\ell$ induces an inclusion of Galois groups $\iota_\ell^*: \op{G}_\ell \hookrightarrow \op{G}_\Q$ where $\op{G}_\ell:=\op{Gal}(\bar{\Q}_\ell/\Q_\ell)$ and $\op{G}_\Q:=\op{Gal}(\bar{\Q}/\Q)$. Let $E/\mathbb{Q}$ be an elliptic curve of conductor $N = N_{E}$ and $p\geq 5$ be a prime at which $E$ has good ordinary reduction. Let $E[p^n]$ (resp. $E[p^\infty]$) be the $p^n$ torsion ($p$-primary torsion) subgroup of $E(\bar{\Q})$, viewed as a module over $\op{G}_\Q$. Set $T:=\varprojlim_n E[p^n]$ to denote the $p$-adic Tate module of $E$. Let $K=\Q(\sqrt{-D})$ be an imaginary quadratic field in which $p$ splits. Write $p\mathcal{O}_K = \mathfrak{p} \overline{\mathfrak{p}}$ and $N=N^+N^-$, where $N^+$ (resp. $N^-$) is the product of prime factors which split (resp. are inert) in $K$. Throughout this section, assume the \emph{generalized Heegner hypothesis}: \begin{itemize}
\item[(\mylabel{GHH}{\textbf{GHH}})]  $N^-$ is a squarefree product of an even number of prime numbers.
\end{itemize} Some of the results stated below will require the stronger assumption that $N^- = 1$. However, this condition will not be imposed in the general discussions that follow.

\par Let $\Sigma_0=\Sigma_0(E)$ be the set of prime numbers $\ell$ which divide $Np$, and let $\Sigma$ be a finite set of rational primes containing $\Sigma_0$. Let $\Q_\Sigma\subset \bar{\Q}$ be the maximal algebraic extension of $\Q$ in which all primes $\ell\notin \Sigma$ are unramified. Given a number field $F \subset \Q_\Sigma$, I shall, for ease of notation, write
\[H^i(\Q_\Sigma/F, \cdot) := H^i(\operatorname{Gal}(\Q_\Sigma/F), \cdot) \quad \text{and} \quad H^i(\Q_\ell, \cdot) := H^i(G_\ell, \cdot).\]
For $\ell \in \Sigma$ and any algebraic extension $F/\Q$, denote by $\Omega_\ell(F)$ the set of primes $v$ of $F$ such that $v \mid \ell$. If $F$ is a number field and $v\in \Omega_\ell(F)$, then let $F_v$ be the completion of $F$ at $v$. On the other hand, if $F/\Q$ is an infinite algebraic extension, then $F_v$ is defined to be the union of completions $F'_v$ for all number fields $F'$ contained in $F$. If $F$ is a number field, set
$$
J_\ell(E/F) := \bigoplus_{v \in  \Omega_\ell(F)} H^1(F_v, E)[p^\infty],
$$
where $E[p^\infty]$ is the $p$-primary torsion subgroup of $E(\bar{\Q})$. On the other hand, if $F$ is an infinite extension, then the corresponding local condition is defined as the direct limit over finite subextensions $F' \subset F$ of $F$:
$$
J_\ell(E/F) := \varinjlim_{F'} J_\ell(E/F'),
$$
where the transition maps in the direct system are the natural restriction maps.
\begin{definition}
Define the $p^\infty$-Selmer group of $E$ over any extension $L\subset \Q_\Sigma$, finite or infinite, as follows
$$
\operatorname{Sel}_{p^\infty}(E/L) := \ker \left( H^1(\Q_\Sigma/L, E[p^\infty]) \longrightarrow \prod_{\ell \in \Sigma} J_\ell(E/L) \right).
$$
\end{definition}

\par Let $K_{\operatorname{cyc}}$ denote the cyclotomic $\mathbb{Z}_p$-extension of $K$. This is the unique $\mathbb{Z}_p$-extension of $K$ contained in $K(\mu_{p^\infty})$. Denote by $K_{\operatorname{ac}}$ the anticyclotomic $\mathbb{Z}_p$-extension of $K$, i.e., the unique $\mathbb{Z}_p$-extension of $K$ that is Galois over $\mathbb{Q}$ and on which complex conjugation acts nontrivially on $\operatorname{Gal}(K_{\operatorname{ac}}/K)$. These two extensions are the only $\mathbb{Z}_p$-extensions of $K$ that are Galois over $\mathbb{Q}$. 

\begin{definition}
    For $n\geq 0$ and $\ast\in \{\rm{cyc}, \rm{ac}\}$, let $K_{\ast,n}$ is the subextension of $K_\ast$ for which $[K_{\ast,n}:K]=p^n$. Define the compact Selmer group at the $n$-th layer by \[\mathcal{S}el(T/K_{\ast, n}):=\varprojlim_m \op{Sel}_{p^m}(E/K_{\ast, n}).\] The compact Selmer group over $K_\ast$ is then defined as follows:
    \[\mathcal{S}el(T/K_\ast):=\varprojlim_n \mathcal{S}el(T/K_{\ast, n}).\]
\end{definition}
\par Let $K_\infty$ be the compositum $K_{\operatorname{cyc}} K_{\operatorname{ac}}$, and set $\operatorname{G}_\infty := \operatorname{Gal}(K_\infty/K)$. Note that $K_\infty$ is the unique $\mathbb{Z}_p^2$-extension of $K$. Given a subgroup $H$ of $\op{G}_\infty$, let $\overline{H}$ be the closure of $H$ in $\op{G}_\infty$. Fix topological generators $\sigma$ and $\tau$ of $G_\infty$ such that
$$
\begin{aligned}
\overline{\langle \sigma \rangle} &= \ker\left( G_\infty \to \op{Gal}(K_{\op{cyc}} / K) \right), \\
\overline{\langle \tau \rangle} &= \ker\left( G_\infty \to \op{Gal}(K_{\op{ac}} / K) \right),
\end{aligned}
$$
i.e., $\sigma$ projects to a generator of $\op{Gal}(K_{\op{cyc}} / K)$ and is trivial on $K_{\op{ac}}$, while $\tau$ projects to a generator of $\op{Gal}(K_{\op{ac}} / K)$ and is trivial on $K_{\op{cyc}}$. Let $\Lambda_\infty$ be the Iwasawa algebra $\mathbb{Z}_p\llbracket G_\infty\rrbracket$ which is identified with $\mathbb{Z}_p\llbracket X, Y\rrbracket$ where $X:=\sigma-1$ and $Y:=\tau-1$. Let $\mathcal{K} / K$ be a $\mathbb{Z}_p$-extension. Then there exists a unique point $(a , b) \in \mathbb{P}^1(\mathbb{Z}_p)$ such that $
\overline{\langle \sigma^a \tau^b \rangle} = \ker\left(G_\infty \to \operatorname{Gal}(\mathcal{K} / K)\right)
$. The $\Z_p$-extension associated to $(a,b)$ is denoted $K_{a,b}$. Define \[\Gamma_{a, b} = \operatorname{Gal}(K_{a, b} / K), \quad H_{a, b} = \operatorname{Gal}(K_\infty / K_{a, b}),\quad\Lambda_{a, b} = \mathbb{Z}_p\llbracket \Gamma_{a, b} \rrbracket.\] Let $\pi_{a, b} : \Lambda \to \Lambda_{a, b}$ denote the map induced by the natural projection $G_\infty \to \Gamma_{a, b}$. Set $f_{a, b} = (1 + X)^a (1 + Y)^b - 1$. Define $\Gamma_{\op{cyc}} := \Gamma_{1,0}$, $\Gamma_{\op{ac}} := \Gamma_{0,1}$ and associated Iwasawa algebras\[\Lambda_{\op{cyc}} := \Lambda_{1,0}=\Z_p\llbracket \Gamma_{\op{cyc}}\rrbracket \text{ and } \Lambda_{\op{ac}} :=\Lambda_{0,1}=\Z_p\llbracket \Gamma_{\op{ac}}\rrbracket.\]The specialization $(a , b) = (1 , 0)$ gives $\Lambda_{\op{cyc}} = \mathbb{Z}_p\llbracket Y \rrbracket$, and $(a , b) = (0 , 1)$ yields $\Lambda_{\op{ac}} = \mathbb{Z}_p\llbracket X \rrbracket$.

\subsection{Iwasawa invariants of $\Z_p\llbracket T\rrbracket$-modules}

\par Consider the $1$-variable Iwasawa algebra $\Lambda:=\Z_p\llbracket T\rrbracket$ and let $M$ be a module over $\Lambda$. The \emph{Pontryagin dual} of $M$ is defined by $M^\vee := \operatorname{Hom}_{\Z_p}(M, \Q_p/\Z_p)$. It is said that $M$ is \emph{cofinitely generated} (respectively, \emph{cotorsion}) over $\Lambda$ if $M^\vee$ is finitely generated (respectively, torsion) as a $\Lambda$-module. A monic polynomial $f(T) \in \Z_p\llbracket T \rrbracket$ is called \emph{distinguished} if all of its non-leading coefficients are divisible by $p$. Let $M$ and $M'$ be cofinitely generated and cotorsion $\Lambda$-modules. The modules $M$ and $M'$ are \emph{pseudo-isomorphic} if there exists a $\Lambda$-module homomorphism $\phi: M \to M'$ whose kernel and cokernel are finite. The structure theorem for finitely generated torsion $\Lambda$-modules implies that any cofinitely generated and cotorsion $\Lambda$-module $M$ is pseudo-isomorphic to a module $M'$ whose Pontryagin dual decomposes as a finite direct sum of cyclic torsion $\Lambda$-modules:
\begin{equation} \label{structure isomorphism}
(M')^\vee \simeq \left( \bigoplus_{i=1}^s \frac{\Z_p\llbracket T\rrbracket}{(p^{n_i})} \right) \oplus \left( \bigoplus_{j=1}^t \frac{\Z_p\llbracket T\rrbracket}{(f_j(T))} \right),
\end{equation}
where $s, t \geq 0$, $n_i \in \Z_{\geq 1}$, and each $f_j(T)$ is a distinguished polynomial. Based on this decomposition, the Iwasawa invariants of $M$ are defined by

$$
\mu_p(M) := \sum_{i=1}^s n_i \quad \text{and} \quad \lambda_p(M) := \sum_{j=1}^t \deg f_j.
$$
\noindent By convention, if $s = 0$ (respectively, $t = 0$), the sum $\sum_{i=1}^s n_i$ (respectively, $\sum_{j=1}^t \deg f_j$) is taken to be zero. The Iwasawa invariants are well defined, i.e., independent of the decomposition \eqref{structure isomorphism}.

\begin{lemma} \label{basic lemma on lambda and corank}
Let $M$ be a cofinitely generated and cotorsion $\Z_p\llbracket T \rrbracket$-module. Then the following conditions are equivalent:
\begin{enumerate}
\item The module $M$ is cotorsion and $\mu_p(M) = 0$.
\item The module $M$ is cofinitely generated as a $\Z_p$-module,
\item $M[p]$ is finite.
\end{enumerate}
\end{lemma}

\begin{proof}
The statement follows directly from the structure theorem \eqref{structure isomorphism}.
\end{proof}

\noindent If $\op{Sel}_{p^\infty}(E/K_{a,b})$ is cotorsion over $\Lambda_{a,b}$, define $\mu_{a,b}$ (resp. $\lambda_{a,b}$) to be its $\mu$-invariant (resp. $\lambda$-invariant). 

\subsection{Greenberg Selmer groups and $p$-congruences}
\par In order to study the effect of $p$-congruences between elliptic curves and modular forms, I recall Greenberg's Selmer groups associated to modular forms. Let $f$ be a normalized $p$-ordinary Hecke eigencuspform of level $N$ and weight $k\geq 2$. Assume that $p\nmid N$. The results in this section will require the \emph{Heegner hypothesis}:
\begin{itemize}
\item[(\mylabel{HH}{\textbf{HH}})]  $N^-=1$, i.e., all primes dividing $N$ split in $K$.
\end{itemize}
\noindent Likewise, I say that an elliptic curve $E$ satisfies \eqref{HH} if the associated modular form does. This condition implies in particular that all primes dividing $Np$ are finitely decomposed in $K_\infty$. Let $F_f$ be the field generated by the Fourier coefficients of $f$ and let $\mathbb{K}$ be the $p$-adic completion of $\iota_p(F_f)\subset \bar{\Q}_p$. I let $\cO$ be the valuation ring of $\mathbb{K}$ and $\varpi$ denote its uniformizer. Denote by $\kappa$ the residue field of $\cO/(\varpi)$. Let $\rho_f: \op{G}_{\Q}\rightarrow \op{Aut}(V_f)\xrightarrow{\sim} \op{GL}_2(\mathbb{K})$ be  the Galois representation associated to $f$. Choose a Galois stable $\cO$-lattice $T_f$ contained in $\cO$ and set $A_f:=V_f/T_f$. Note that in the special case when $f$ has weight $2$ and $F_f=\Q$, the $p$-divisible group $A_f$ equals $E[p^\infty]$, where $E$ is an elliptic curve over $\Q$. I let $\Sigma$ be a finite set of primes containing the primes dividing $Np$ such that all primes $\ell\in \Sigma$ are split in $K$. By class field theory, there is a unique $\Z_p$-extension of $K$ in which $\p$ (resp. $\bar{\p}$) is unramified. Let $\cK=K_{a,b}$ be a $\Z_p$-extension of $K$ in which both the primes $\mathfrak{p}$ and $\bar{\mathfrak{p}}$ are totally ramified. For $\ell\neq p$, set \[\cH_\ell(E/\cK):=\bigoplus_{v\in \Omega_\ell(\cK)} H^1(\cK_v, E[p^\infty]).\] Since $f$ is $p$-ordinary, it follows that for each $w \in \{\p, \bar{\p}\}$, there exists a short exact sequence of $\operatorname{G}_w$-modules
$$
0 \rightarrow A_f^+ \rightarrow A_f \rightarrow A_f^- \rightarrow 0,
$$ where $A_f^+$ denotes the unramified submodule on which $\operatorname{G}_w$ acts via the unramified quotient. Note that if $A_f=E[p^\infty]$, the above sequence coincides with 
\[0\rightarrow \widehat{E}[p^\infty]\rightarrow E[p^\infty]\rightarrow \widetilde{E}[p^\infty]\rightarrow 0, \] where $\widetilde{E}$ is the reduced curve and $\widehat{E}$ is the formal group of $E$. Let $w'|w$ be the prime of $\cK$ which lies above $w$. Denote by $\op{I}_{w'}$ the inertia subgroup of $\op{G}_{w'}$ and let 
\[\pi_w: H^1(\cK_{w'}, A_f)\longrightarrow H^1(\op{I}_{w'}, A_f^-)\] be the composite of maps:
\[H^1(\cK_{w'}, A_f)\rightarrow H^1(\cK_{w'}, A_f^-)\rightarrow H^1(\op{I}_{w'}, A_f^-),\] where the first map is induced by $E\rightarrow \widetilde{E}$ and the second is the restriction map to $\op{I}_{w'}$. Then, the Greenberg condition at $p$ is defined as follows
\[\cH_p(E/\cK):=\frac{H^1(\cK_{\p}, A_f)}{\op{ker} \pi_{\p}}\times \frac{H^1(\cK_{\bar{\p}}, A_f)}{\op{ker} \pi_{\bar{\p}}}.\]
\begin{definition}
    The Greenberg Selmer group is defined as follows:
    \[\op{Sel}_{\op{Gr}}(A_f/\cK):=\op{ker}\left\{ H^1(\Q_\Sigma/\cK, A_f)\longrightarrow \bigoplus_{\ell\in \Sigma} 
 \cH_\ell(A_f/\cK)\right\}.\]
\end{definition}
\noindent For an elliptic curve $E_{/\Q}$, the Greenberg Selmer group $\op{Sel}_{\op{Gr}}(E[p^\infty]/\cK)$ coincides with $\op{Sel}_{p^\infty}(E/\cK)$. In fact, the local conditions defining these Selmer groups coincide over $\cK$, provided the primes $\p$ and $\bar{\p}$ are totally ramified in $\cK$. I define a Selmer group associated with the residual representation $A_f[\varpi]$, which is denoted by $\op{Sel}_{\op{Gr}}(A_f[\varpi]/\cK)$. If $\ell\neq p$, then \[\cH_\ell(A_f[\varpi]/\cK):=\bigoplus_{v\in \Omega_\ell(\cK)} H^1(\cK_v, A_f[\varpi]).\] On the other hand, \[\cH_p(A_f[\varpi]/\cK):=\frac{H^1(\cK_{\p}, A_f[\varpi])}{\op{ker} \bar{\pi}_{\p}}\times \frac{H^1(\cK_{\bar{\p}}, A_f[\varpi])}{\op{ker} \bar{\pi}_{\bar{\p}}},\] where \[\bar{\pi}_w: H^1(\cK_{w'}, A_f[\varpi])\longrightarrow H^1(\op{I}_{w'}, A_f^-[\varpi]).\]   

\begin{definition}
    The Greenberg Selmer group associated with $A_f[\varpi]$ is defined as follows:
    \[\op{Sel}_{\op{Gr}}(A_f[\varpi]/\cK):=\op{ker}\left\{ H^1(\Q_\Sigma/\cK, A_f[\varpi])\longrightarrow \bigoplus_{\ell\in \Sigma} 
 \cH_\ell(A_f[\varpi]/\cK)\right\}.\]
\end{definition}

\begin{proposition}\label{basic mu invariant p torsion lemma}
    Assume that $H^0(K, A_f[\varpi])=0$. Then the following are equivalent:
    \begin{enumerate}
        \item $\op{Sel}_{\op{Gr}}(A_f/\cK)$ is cotorsion over $\Z_p\llbracket T\rrbracket$ with $\mu=0$,
         \item $\op{Sel}_{\op{Gr}}(A_f/\cK)[p]$ is finite,
        \item $\op{Sel}_{\op{Gr}}(A_f[\varpi]/\cK)$ is finite.     
    \end{enumerate}
\end{proposition}
\begin{proof}
    \par The argument is similar to those in Section 2 of \cite{greenbergvatsal00}. For the convenience of the reader, I provide a sketch of the proof here. That (1) and (2) are equivalent follows from Lemma \ref{basic lemma on lambda and corank}. I prove that (2) and (3) are equivalent. It is clear that (2) is equivalent to the condition that $\op{Sel}_{\op{Gr}}(A_f/\cK)[\varpi]$ is finite. There is a natural commutative diagram:\begin{equation}\label{fdiagram}
\begin{tikzcd}[column sep = small, row sep = large]
0\arrow{r} & \op{Sel}_{\op{Gr}}(A_f[\varpi]/\cK) \arrow{r}\arrow{d}{\alpha} & H^1(\Q_\Sigma/\cK, A_f[\varpi])\arrow{r} \arrow{d}{g} & \bigoplus_{\ell\in \Sigma} \cH_\ell(A_f[\varpi]/\cK) \arrow{d}{h} \\
0\arrow{r} & \op{Sel}_{\op{Gr}}(A_f/\cK)[\varpi] \arrow{r} & H^1(\Q_\Sigma/\cK, A_f)[\varpi] \arrow{r}  &\bigoplus_{\ell\in \Sigma} \cH_\ell(A_f/\cK)[\varpi],
\end{tikzcd}
\end{equation}
where $h$ is the direct sum of natural maps 
\[h_\ell: \cH_\ell(A_f[\varpi]/\cK)\rightarrow \cH_\ell(A_f/\cK)[\varpi]\] induced from the Kummer sequence. The primes $\ell\in \Sigma$ are finitely decomposed in $K_\infty$ and it is easy to see that the kernel of $h_\ell$ is finite for every $\ell\in \Sigma$. Since $\cK/K$ is a pro-$p$ extension, the assumption that $H^0(K, A_f[\varpi])=0$ implies that $H^0(\cK, A_f[\varpi])=0$ (see \cite[Corollary 1.6.13]{NSW}). From the Kummer sequence, the map $g$ is an isomorphism. This implies particular that $\op{ker}\alpha=0$ and $\op{cok}\alpha$ injects into the kernel of $h$. The map $h$ has finite kernel and thus the cokernel of $\alpha$ is finite. Thus, I have shown that (2) and (3) are equivalent.
\end{proof}

\begin{proposition}\label{congruence propn}
    Let $\cK/K$ be a $\Z_p$-extension with Iwasawa algebra $\Z_p\llbracket T\rrbracket$. Suppose that $E_{/\Q}$ is an elliptic curve and $f$ is a Hecke eigencuspform with associated module $A_f$. Let $p$ be an odd prime which splits in $K$, at which $E$ (resp. $f$) has good ordinary reduction (resp. is $p$-ordinary). Assume moreover that $E[p]\otimes_{\F_p} \kappa$ is isomorphic to $A_f[\varpi]$ as $\op{G}_{\Q}$-modules. Assume moreover that:
    \begin{enumerate}
        \item $E(K)[p]=0$ and $H^0(K, A_f[\varpi])=0$,
        \item $E$ and $f$ both satisfy the Heegner hypothesis \eqref{HH}.
    \end{enumerate}
     Then, $\op{Sel}_{p^\infty}(E/\cK)$ is cotorsion as a $\Z_p\llbracket T \rrbracket$-module with $\mu=0$ if and only if $\op{Sel}_{\op{Gr}}(A_f/\cK)$ is cotorsion as a $\Z_p\llbracket T \rrbracket$-module with $\mu=0$.
\end{proposition}

\begin{proof}
    Let $N$ (resp. $N'$) be the conductor (resp. level) of $E$ (resp. $f$). Denote by $\Sigma$ the set of rational primes $\ell|NN'p$. It is with respect to this choice of $\Sigma$ that the residual Selmer groups $\op{Sel}_{\op{Gr}}(E[p]/\cK)$ and $\op{Sel}_{\op{Gr}}(A_f[\varpi]/\cK)$ are defined. By hypothesis, every prime in $\Sigma$ is split in $K$, hence finitely decomposed in $K_\infty$. Since $E[p]\otimes_{\F_p} \kappa$ and $A_f[\varpi]$ are isomorphic as $\op{G}_{\Q}$-modules, it follows that $\op{Sel}_{\op{Gr}}(E[p]/\cK)$ is finite if and only if $\op{Sel}_{\op{Gr}}(A_f[\varpi]/\cK)$ is finite. The result then follows from Proposition \ref{basic mu invariant p torsion lemma}.
\end{proof}
Assume that $\op{corank}_{\Z_p} \op{Sel}_{p^\infty} (E/K)=1$; I recall a numerical criterion due to Perrin--Riou \cite{PR82} and Schneider \cite{Schneider85} for computing the $\mu$-invariant and $\lambda$-invariant of the Selmer group $\op{Sel}_{p^\infty}(E/K_{\rm{cyc}})$. I introduce conditions on a triple $(E, K, p)$:
\begin{itemize}
\item[(\mylabel{tEC}{\textbf{tEC}})] \begin{itemize}
        \item $\op{corank}_{\Z_p} \op{Sel}_{p^\infty} (E/K)=1$,
        \item $p\nmid \mathcal{R}_p(E/K)$,
        \item $\# \Sh(E/K)[p^\infty]=0$,
        \item $p\nmid c_v(E/K)$ for all primes $v\nmid p$ of $K$,
        \item and $\widetilde{E}(\kappa_v)[p^\infty]=0$ for all primes $v|p$ of $K$.
    \end{itemize}  
\end{itemize}
\begin{theorem}\label{PR and S}
With respect to notation above, the following conditions are equivalent.
\begin{enumerate}
    \item The Iwasawa invariants of $\op{Sel}_{p^\infty}(E/K_{\op{cyc}})$ are given by \[\mu\left(\op{Sel}_{p^\infty}(E/K_{\op{cyc}})\right)=0\text{ and }\lambda\left(\op{Sel}_{p^\infty}(E/K_{\op{cyc}})\right)=1;\]
    \item the condition \eqref{tEC} is satisfied for $E_{/K}$.
\end{enumerate}
\end{theorem}
\begin{proof}
I refer to \cite[Corollary 2.17]{muller2024hilbert} for the proof of the result.
\end{proof}

\section{Propagating the growth number conjecture via congruences}

\subsection{Prior results}
\par I recall some prior results related to Mazur's growth number conjecture, which will be instrumental in the proofs of the main results. These results make use of the structure of Selmer groups and height pairings in the anticyclotomic Iwasawa setting. The generalized Heeger hypothesis implies that there is a compatible system of Heegner classes $z_n\in \mathcal{S}el(T/K_{\rm{ac},n})$, whose inverse limit gives rise to a class 
\[z_\infty =(z_n)_{n \ge 0}\in \cS el(T/K_{\rm{ac}}).\] 
Let $\mathcal{J}$ denote the augmentation ideal of $\Lambda_{\cyc}$. Perrin--Riou \cite{PR} defines a height pairing, which for compact Selmer groups at the \emph{n-th level} is given by \[\langle \cdot, \cdot \rangle_n: \mathcal{S}el(T/K_{\rm{ac},n})\times \mathcal{S}el(T/K_{ \rm{ac},n})\longrightarrow p^{-k} \Z_p\otimes_{\Z_p} \mathcal{J}/\mathcal{J}^2,\]
where $k\geq 0$ is an integer, which is independent of $n$. These pairings at finite layers are then interpolated into a $\Lambda_{\rm{ac}}$ height pairing:
\[
\langle \cdot , \cdot \rangle_{\Lambda_{\ac}} : \cS el (T/K_{\rm{ac}}) \times \cS el (T/K_{\rm{ac}}) \longrightarrow \mathbb{Q}_p \otimes_{\mathbb{Z}_p} \Lambda_{\ac} \otimes_{\mathbb{Z}_p} \mathcal{J} / \mathcal{J}^2,\] which is defined as follows:
\[
\langle a_\infty , b_\infty \rangle_{\Lambda_{\ac}} = \varprojlim_{n} \sum_{\theta \in \op{Gal}(K_{\ac,n}/K)} \langle a_n , b_n^\theta \rangle_{n}\cdot \theta.
\]

By the work of Hida \cite{Hida2var2} and Perrin--Riou \cite{PRHeegner, PRpadicLfunction}, there exists a $2$-variable $p$-adic L-function $L_p(X, Y)\in \Lambda_\infty \otimes_{\Z_p} \Q_p$. The following result is a culmination of work of Howard \cite{HowardIwasawatheoreticGZ}, Castella \cite{Castella} and Disegni \cite{disegni1, Disegni2}.

\begin{theorem}
Let $p \geq 5$ be an odd prime number and $E_{/\Q}$ be an elliptic curve with good ordinary reduction at $p$. Let $K$ be an imaginary quadratic field such that the prime $p$ splits in $K$. Assume that $N^-$ is the squarefree product of an even number of primes. Then, 
\[\Big\langle \frac{\partial L_p}{\partial Y}(X,0) \Big\rangle = \Big\langle \langle z_\infty, z_\infty \rangle_{\Lambda_{\ac}} \Big\rangle\]
as principal ideals in $\Q_p\otimes \Lambda_{\rm{ac}}$.
\end{theorem}
\begin{proof}
    The result follows from \cite[Theorem C (4)]{disegni1}; see also \cite[Theorem 2.5]{kundulei2025} and the references therein.
\end{proof}
\par I shall discuss the main result in \cite{kundulei2025}, and the necessary hypotheses. Let $I\subset \Lambda_\infty$ denote the characteristic ideal of the dual Selmer group $\op{Sel}_{p^\infty}(E/K_\infty)^\vee$. The Iwasawa main conjecture predicts the following equality: 
\begin{equation}
\label{IMC}
\tag{\textbf{IMC}} \langle L_p(X,Y) \rangle = I.     
\end{equation}
\noindent Let $\rho_E:\op{G}_{\Q}\rightarrow \op{GL}_2(\F_p)$ be the Galois representation on $E[p]$, and for a prime $\ell$, denote by $\op{I}_\ell$ the inertia subgroup of $\op{G}_{\Q_\ell}$. The set of \emph{vexing primes} of $E[p]$ is defined as follows 
\[\mathcal{V}:=\{\ell\equiv -1\pmod{p}\mid \bar{\rho}_{E|\op{G}_{\Q_\ell}}\text{ is irreducible, and } \bar{\rho}_{E|\op{I}_\ell}\text{ is reducible}\}.\]
\begin{theorem}[Burungale, Castella, Skinner]\label{BCS-IMC}
    Let $E_{/\Q}$ be an elliptic curve and $p\geq 5$ be a prime at which $E$ has good ordinary reduction. Let $K$ be an imaginary quadratic field in which $p$ splits. Assume that the following conditions are satisfied:
    \begin{enumerate}
        \item $\rho_E$ is surjective.
        \item The discriminant $D_K$ of $K$ is odd, and $D_K\neq -3$.
        \item Assume that $E_{/K}$ satisfies the Heegner hypothesis \eqref{HH}.
        \item The set of vexing primes $\mathcal{V}$ is empty.
    \end{enumerate}
    Then, $(E_{/K}, p)$ satisfies \eqref{IMC}, i.e., 
 \[\langle L_p(X,Y) \rangle = \op{char}_{\Lambda_\infty}\left(\op{Sel}_{p^\infty}(E/K_\infty)^\vee\right).\]
\end{theorem}

\begin{proof}
    This result is \cite[Theorem 1.4.1]{BCS}.
\end{proof}
 
 The results in \cite{kundulei2025} require one inclusion of the Iwasawa main conjecture:
\begin{equation}
    \label{h-IMC} \tag{\textbf{h-IMC}} L_p(X,Y) \in I.
\end{equation}

\begin{theorem}[Kundu, Lei]\label{kundu lei theorem}
Let $p\geq 5$ be a prime and $E_{/\Q}$ be an elliptic curve with analytic rank $1$, without complex multiplication. Let $K$ be an imaginary quadratic field where $p$ splits. Assume that:
\begin{enumerate}
    \item $N^-$ is the squarefree product of an even number of primes, 
    \item $E(K)[p]=0$, 
    \item one divisibility of the main conjecture \eqref{h-IMC} holds, 
    \item $\langle z_\infty, z_\infty\rangle_0\neq 0$.
\end{enumerate}
Then, Conjecture \ref{Mazur-GNC} holds for $E_{/K}$.
\end{theorem}

\subsection{Main results}

\par In this section, $p\geq 5$ and $E_{/\Q}$ be an elliptic curve with good ordinary reduction at $p$, such that $E_{/K}$ satisfies \eqref{HH}, \eqref{h-IMC} and \eqref{tEC}. I show that these conditions imply that the Selmer group $\op{Sel}_{p^\infty}(E/K_\infty)$ is semisimple as a module over $\Lambda_\infty$.

\begin{definition}A module $M$ is said to satisfy (\mylabel{hyp:SC}{\textbf{S-C}}) if it is a direct sum of cyclic torsion $\Lambda_\infty$-modules.
\end{definition}
\noindent When $M$ satisfies \eqref{hyp:SC}, the characteristic ideal of $M$ is given by $\op{char}(M):=\prod I_i$ where $M\simeq \bigoplus_i \frac{\Lambda_\infty}{(I_i)}$. In this case, $\op{char}(M)$ is necessarily principal, cf. \cite[Theorem 2.4]{GHKL}. 

\begin{proposition}
\label{prop:SC-ord}
Assume that $\lambda\left(\operatorname{Sel}_{p^\infty}(E/K_{\mathrm{cyc}})^\vee\right) \leq 1$ and $\mu\left(\operatorname{Sel}_{p^\infty}(E/K_{\mathrm{cyc}})^\vee\right) = 0$. Then $\operatorname{Sel}_{p^\infty}(E/K_\infty)^\vee$ satisfies \eqref{hyp:SC}.
\end{proposition}

\begin{proof}
    See \cite[Proposition 5.1]{GHKL} for a proof of this result.
\end{proof}

\begin{proposition}\label{control propn}
   The natural restriction map restriction map then induces an isomorphism
\[
\op{Sel}_{p^\infty}(E/K_{a,b})\simeq \op{Sel}_{p^\infty}(E/K_\infty)^{H_{a,b}}.
\]
\end{proposition}
\begin{proof}
    See \cite[Proposition 3.2]{GHKL}.
\end{proof}
\noindent Denote by $\cL_{a,b}:=\pi_{a,b}\left( L_p(X,Y)\right)$ and set $I_{a,b}:=\pi_{a,b}(I)$. Given $a, b\in \Q_p$, I write $a\sim b$ if $a=ub$ where $u\in \Z_p^\times$. 

\begin{proposition}\label{main long propn}
    Let $E_{/\Q}$ be an elliptic curve with good ordinary reduction at a prime $p\geq 5$ and $K$ be an imaginary quadratic number field. Assume that $E(K)[p]=0$ and $(E,K,p)$ satisfies the conditions \eqref{HH}, \eqref{IMC} and \eqref{tEC}. Then, the following assertions hold.
    \begin{enumerate}
    \item For each $(a,b)\in \mathbb{P}^1(\Z_p)$ such that $(a,b)\neq (0,1)$, the Selmer group $\op{Sel}_{p^\infty}(E/K_{a,b})$ is cotorsion over $\Lambda_{a, b}$ with $\mu_{a,b}=0$ and $\lambda_{a,b}=1$.
    \item Conjecture \ref{Mazur-GNC} holds for $(E,K,p)$. 
    \end{enumerate}
\end{proposition}
\begin{proof}
    Write $\Lambda_{1,0}=\Lambda_{\rm{cyc}}$ as a formal power series ring $\Z_p\llbracket T\rrbracket$. As a consequence of Proposition \ref{control propn}, one finds that $I_{a,b}$ is the characteristic ideal of of $\op{Sel}_{p^\infty}(E/K_{a,b})^\vee$. The hypothesis \eqref{IMC} implies that $(\cL_{a,b})=I_{a,b}$. On the other hand, Theorem \ref{PR and S} in conjunction with hypothesis \eqref{tEC} implies that the $\mu$-invariant (resp. $\lambda$-invariant) of $I_{1,0}$ is $0$ (resp. $1$). As a consequence, $I_{1,0}=(T)$. On the other hand, 
    \[\cL_{1,0}(T)= \alpha T+\cdots \text{higher order terms},\] where $\alpha\sim \langle z_\infty, z_\infty\rangle_0$, see for instance \cite[proof of Theorem 3.1]{kundulei2025}. Since $\cL_{1,0}$ divides $(T)$, I deduce that $\langle z_\infty, z_\infty\rangle_0\in \Z_p^\times$. Now let $(a,b)\in \mathbb{P}^1(\Z_p)$ be such that $(a,b)\neq (0,1)$. Writing $\Lambda_{a,b}=\Z_p\llbracket T_{a,b}\rrbracket$. By the same arguments in \emph{loc. cit.} I may express $\cL_{a,b}(T_{a,b})$ as 
    \[\cL_{a,b}(T_{a,b})=\alpha_{a,b} T+\dots +\text{higher order terms},\] where $\alpha_{a,b}\sim \langle z_\infty, z_\infty\rangle_0$. Thus, I find that $\alpha_{a,b}\in \Z_p^\times$ and this implies that $I_{a,b}=(T)$. I thus deduce that $\mu_{a,b}=0$ and $\lambda_{a,b}=1$ and this completes the proof of (1). 
    \par Since $\langle z_\infty, z_\infty\rangle_\infty\neq 0$, it follows from Theorem \ref{kundu lei theorem} that Conjecture \ref{Mazur-GNC} holds for $(E,K,p)$.
\end{proof}

\begin{theorem}\label{main theorem of paper}
Let $E$ be an elliptic curve over $\Q$ with good ordinary reduction at a prime $p\geq 5$ and $K$ be an imaginary quadratic number field. Assume that $E_{/K}$ satisfies the conditions \eqref{HH}, \eqref{IMC} and \eqref{tEC}. Let $f$ be a $p$-ordinary Hecke eigencuspform with $p\nmid N_f$. Assume that $f$ satisfies \eqref{HH} and $A_f[\varpi]\simeq E[p]\otimes_{\F_p}\kappa$ as modules over $\op{G}_\Q$. Then, for each $(a,b)\neq (0,1)$ such that $\p$ and $\bar{\p}$ are totally ramified in $K_{a,b}$, the $\Z_p$-corank of $\op{Sel}_{\op{Gr}}(A_f/K_{a,b})$ is finite. In particular, if $A_f=E'[p^\infty]$ for some elliptic curve $E'$ over $\Q$, the rank of $E'(K_{a,b})$ is finite.
\end{theorem}

\begin{proof}
    For any $(a,b)\in \mathbb{P}^1(\Z_p)$ such that $(a,b)\neq (0, 1)$, Proposition \ref{main long propn} implies that $\op{Sel}_{p^\infty}(E/K_{a,b})$ is cotorsion over $\Lambda_{a, b}$ with $\mu_{a,b}=0$ and $\lambda_{a,b}=1$. It then follows from Proposition \ref{congruence propn} that $\op{Sel}_{p^\infty} (A_f/K_{a,b})$ is cotorsion over $\Lambda_{a,b}$, and consequently the $\Z_p$-corank of $\op{Sel}_{\op{Gr}}(A_f/K_{a,b})$ is finite.
\end{proof}
\begin{example}\label{example} I illustrate Theorem \ref{main theorem of paper} through an explicit example. Consider $p=5$ and the elliptic curves defined by their Cremona labels: $E:=37a1$ and $E':=1406g1$. Then, $E$ and $E'$ both have good ordinary reduction at $5$ and as $\op{G}_{\Q}$-modules, one has an isomorphism of the residual representations $E[5]\simeq E'[5]$. The elliptic curves $E_{/K}$ and $E'_{/K}$ satisfy \eqref{HH}. Note that $5$ splits in $K$ and computations on \emph{SAGE} show that $(E, K, p=5)$ satisfies \eqref{tEC}. Since $37\not\equiv -1\pmod{5}$, the set of vexing primes $\mathcal{V}$ for $E$ is empty. The residual representation $\rho_E$ is surjective and hence by Theorem \ref{BCS-IMC}, the \eqref{IMC} is satisfied for $(E, K, p=5)$. Thus Theorem \ref{main theorem of paper} applies to show that $E'(K_{a,b})$ has finite rank for any $\Z_p$-extension $K_{a,b}/K$ in which $\p$ and $\bar{\p}$ are totally ramified.
\end{example}
\subsection{An application involving Hida families}\label{hida subsection}
\par Let $(E, K, p)$ be a triple satisfying the hypotheses of Theorem~\ref{main theorem of paper}. Let $N$ denote the conductor of $E$, and let $\bar{\rho} : \op{Gal}(\overline{\Q}/\Q) \to \op{GL}_2(\F_p)$ be the residual Galois representation arising from the action on $E[p]$. Assume that $\bar{\rho}$ is absolutely irreducible, and let $f_E$ denote the weight two cuspidal eigenform of trivial nebentypus associated with $E$.

\par Let $\bT_N$ be the universal reduced ordinary Hecke algebra of tame level $N$, $p$-power level, and arbitrary weight. Let $\Gamma' \subset \Z_p^\times$ denote the group of principal units, and define $\Lambda' := \Z_p\llbracket \Gamma' \rrbracket$, the corresponding Iwasawa algebra. After fixing a topological generator of $\Gamma'$, I identify $\Lambda'$ with the formal power series ring $\Z_p\llbracket X \rrbracket$. Let $\bT_N^{\mathrm{new}}$ denote the new quotient of $\bT_N$, and let $\mathfrak{m} \subset \bT_N$ be the maximal ideal corresponding to $\bar{\rho}$. Set $\mathcal{T} := \bT_{N,\mathfrak{m}}$ to be the localization of $\bT_N$ at $\mathfrak{m}$. A minimal prime ideal $\mathfrak{a} \subset \mathcal{T}$ is called a \emph{branch}, and I write $\mathcal{T}(\mathfrak{a}) := \mathcal{T}/\mathfrak{a}$. The scheme $\op{Spec} \mathcal{T}(\mathfrak{a})$ is referred to as a \emph{Hida family}. I refer to $\op{Spec} \Lambda'$ as the \emph{weight space} and to $\pi_{\mathfrak{a}} : \op{Spec} \mathcal{T}(\mathfrak{a}) \to \op{Spec} \Lambda'$ as the \emph{map to weight space}. This map is finite and surjective. The Hida family $\op{Spec} \mathcal{T}(\mathfrak{a})$ is equipped with a Galois representation
\[\rho_{\mathfrak{a}} : \op{G}_\Q \to \op{GL}_2(\mathcal{T}(\mathfrak{a}))\]
\noindent lifting $\bar{\rho}$.
\par Now let $\op{Spec} \mathcal{T}(\mathfrak{a})$ be one of the branches which contains the point corresponding to $f_E$. Then for each integer $k > 0$ with $k \equiv 2 \pmod{p-1}$, there exists a classical modular form $f_k$ of weight $k$ lying on $\op{Spec} \mathcal{T}(\mathfrak{a})$, congruent modulo $p$ to $f_E$. The conclusion of Theorem~\ref{main theorem of paper} then implies that for each pair $(a, b) \ne (0,1)$ such that the primes $\p$ and $\bar{\p}$ are totally ramified in the extension $K_{a,b}$, the $\Z_p$-corank of the Greenberg Selmer group $\op{Sel}_{\op{Gr}}(A_{f_k}/K_{a,b})$ is finite.

\bibliographystyle{alpha}
\bibliography{references}
\end{document}